\documentclass[12pt,a4paper]{amsart}
\usepackage[utf8]{inputenc}
\usepackage{amsfonts}
\usepackage{amsmath, amssymb}
\usepackage{amsthm}
\usepackage{cases}
\usepackage[margin=.8in]{geometry}
\usepackage{tkz-euclide}

\newtheorem{theorem}{Theorem}[section]

\newtheorem{defin}[theorem]{Definition}
\newtheorem{lemma}[theorem]{Lemma}

\newtheorem{rem}[theorem]{Remark}

\usepackage{hyperref}
\hypersetup{
  colorlinks = true, 
  urlcolor = blue, 
  linkcolor = blue, 
  citecolor = blue 
}

\let \div \relax

\let \vb \mathbf

\DeclareMathOperator{\div}{div}

\DeclareMathOperator{\dist}{dist}

\DeclareMathOperator{\N}{\mathbb{N}}

\DeclareMathOperator{\R}{\mathbb{R}}

\newcommand{\dd}{\, \mathrm{d}}
\newcommand{\del}{\partial}
\newcommand{\e}{\varepsilon}

\newcommand{\weak}{\rightharpoonup}
\newcommand{\B}{\mathcal{B}}

\newcommand{\vu}{\vb u}
\newcommand{\ms}{\mathbb{S}}

\renewcommand{\rho}{\varrho}
\renewcommand{\phi}{\varphi}

\allowdisplaybreaks

\title{Homogenization of the two-dimensional evolutionary compressible Navier-Stokes equations}
\date{\today}

\author{\v S\'arka Ne\v casov\'a}
\address{Institute of Mathematics, Czech Academy of Sciences,
\v Zitn\'a 25, 115 67 Praha 1, Czech Republic.}
\email{matus@math.cas.cz}

\author{Florian Oschmann}
\address{Institute of Mathematics, Czech Academy of Sciences,
\v Zitn\'a 25, 115 67 Praha 1, Czech Republic.}
\email{oschmann@math.cas.cz}

\begin{document}
\maketitle

\begin{abstract}
We consider the evolutionary compressible Navier-Stokes equations in a two-di\-men\-sional perforated domain, and show that in the subcritical case of very tiny holes, the density and velocity converge to a solution of the evolutionary compressible Navier-Stokes equations in the non-perforated domain.
\end{abstract}

\section{Introduction}
Homogenization of different types of fluid flow models have been extensively investigated during the last decades and are still topic of research. One of the first results in this field was given by Tartar in \cite{Tartar1980} for the case the obstacles' mutual distance is of the same order as their radius. Later, for stationary incompressible Stokes and Navier-Stokes equations, in his seminal PhD thesis \cite{Allaire1990a, Allaire1990b} Allaire figured out that for perforated domains in dimension $d\geq 2$, there are essentially three regimes of particle sizes: very tiny particles shall not influence the flow in a crucial way, meaning that the limiting system is the same as the one in the perforated domain. Large holes put large friction onto the fluid, leading in the limit to Darcy's law. The ``in-between case'' of critically sized holes leads to an additional Brinkman term, which was already discovered for the Poisson equation in \cite{CioranescuMurat1982}. To make things more precise, for $\e>0$ being the distance between particles and dimension $d=2$, Allaire considered holes scaling like $\exp(-\e^{-\alpha})$ for some $\alpha>0$. Here, $\alpha>2$ corresponds to the case of tiny holes, $\alpha<2$ to the case of large holes, and $\alpha=2$ is the critical value. The main difference between those three regimes is, heuristically, the following: small holes with $\alpha>2$ are too tiny to have a high effect on the fluid's velocity, thus, in the limit, the equations preserve their form. For large holes $\alpha<2$, the fluid is hindered more and more until the fluid velocity slows down to zero. By a proper rescaling, one can then recover Darcy's law. The critical case $\alpha=2$ corresponds to a Brinkman law, where in the limiting equations an extra friction term pops up, which is reminiscent of the holes. These heuristics are made rigorous in \cite{Allaire1990a, Allaire1990b}. In this context, let us mention the book \cite{MarchenkoKhruslov2008}, where different types of inclusions and boundary values are discussed, and also the related work \cite{KhrabustovskyiPlum2022}, where the scalar resolvent problem for the Robin-Laplace operator was considered for all types of hole sizes, recovering similar Brinkman-like laws as in \cite{CioranescuMurat1982}.\\

Nowadays, a vast literature on fluid flow homogenization is available. Without claiming completeness, we cite here just a few, ordered to the fluid model they belong to, and refer to the references therein for further reading. Stationary incompressible Stokes equations in two and three spatial dimensions where investigated in \cite{Lu2020} for the case the particles are distributed according to some hard-sphere condition. This condition occurred already in \cite{Allaire1990a, Allaire1990b}, and was relaxed in \cite{Hillairet2018}, where the author considered disjoint holes that may be ``close'' to each other (in a sense to be specified). In \cite{GiuntiHoefer2019, Giunti2021} the authors considered randomly placed holes that are allowed to overlap for the critical and large-size case, respectively. H\"ofer considered in \cite{Hoefer2022} the unsteady system, together with several sizes of holes and a vanishing viscosity limit. For compressible fluids, both stationary and time-dependent Navier-Stokes equations in three dimensions where considered in \cite{Masmoudi2002, FeireislLu2015, DieningFeireislLu2017, LuSchwarzacher2018, HoeferKowalczykSchwarzacher2021, BellaOschmann2022, BellaOschmann2023} also for different sizes and configurations of holes. We also emphasize the works regarding homogenization of three-dimensional Navier-Stokes-Fourier equations in \cite{LuPokorny2021, PokornySkrisovski2021, Oschmann2022}.\\

Regarding homogenization of \emph{two-dimensional} compressible Navier-Stokes equations, to the best of the authors' knowledge, the only available result is \cite{NecasovaPan2022}, which focusses on the steady case. The aim of this paper is therefore to investigate the homogenization of the \emph{evolutionary} system. Note that one of the main difficulties between the three- and two-dimensional case is the use of appropriate cut-off functions in order to pass to the limit in the momentum equation. These functions are related two the harmonic (also called Newtonian) capacity of the holes: whereas in three dimensions, essentially linear cut-offs do the job, two-dimensional cut-off functions must be constructed in a more subtle manner. To show our result, we will make use of an idea of Bravin in \cite{Bravin2022}, where the author introduced a refined concept of cut-off functions for a single hole in $\R^2$. We adopt this strategy and generalize it to the case of many obstacles.\\

\paragraph{\textbf{Notation.}} We use the standard notations for Lebesgue and Sobolev spaces, and denote them even for vector- or matrix-valued functions as in the scalar case, e.g., $L^p(D)$ instead of $L^p(D;\R^2)$. The Frobenius inner product of two matrices $A,B\in \R^{2\times 2}$ is denoted by $A:B=\sum_{i,j=1}^2 A_{ij} B_{ij}$. Moreover, we use the notation $a\lesssim b$ whenever there is a generic constant $C>0$ which is independent of $a$, $b$, and $\e$ such that $a\leq C b$. Lastly, we denote for a function $f$ with domain of definition $D_f\subset\R^2$ its zero prolongation by $\tilde{f}$, that is,
\begin{align*}
\tilde{f}=f \text{ in } D_f,\quad \tilde{f}=0 \text{ in } \R^2\setminus D_f.
\end{align*}

\section{The model, weak solutions, and the main result}
In this section, we introduce the perforated domain, the evolutionary compressible Navier-Stokes equations, and state our main result. We start with the description of the perforated domain and the equations governing the fluid's motion.
\subsection{The perforated domain and the Navier-Stokes equations}
For $\e\in (0,1)$, let $D\subset \R^2$ be a bounded domain with smooth boundary, $\{z_i^\e\}_{i\in \N}\subset \R^2$ be a collection of distinct points, and $K_\e\subset\N$ be the set of indices such that
\begin{align}
\{z_i^\e\}_{i\in K_\e}\subset D,\quad \forall i,j\in K_\e, i\neq j: |z_i^\e-z_j^\e|\geq 2\e,\ \dist(z_i^\e, \del D)>\e. \label{defKeps}
\end{align}
We also assume that the holes may become ``denser'' in $D$ as $\e\to 0$, that is,
\begin{align}\label{scKe}
\exists C>0\, \forall \e>0:\quad |K_\e| \leq C\e^{-2}.
\end{align}
Moreover, let $F\subset B_1(0)$ be a compact, simply connected set with smooth boundary and $0\in F$, $\alpha>2$, and set
\begin{align}
a_\e &= e^{-\e^{-\alpha}},\quad D_\e= D\setminus \bigcup_{i\in K_\e} (z_i^\e+a_\e F). \label{defDeps}
\end{align}
For fixed $T>0$, we consider in $(0,T)\times D_\e$ the evolutionary compressible Navier-Stokes equations
\begin{align}\label{NSE}
\begin{cases}
\del_t \rho_\e + \div(\rho_\e \vu_\e)=0 & \text{in } (0,T)\times D_\e,\\
\del_t(\rho_\e \vu_\e) + \div(\rho_\e \vu_\e \otimes \vu_\e) + \nabla p(\rho_\e) = \div \ms(\nabla \vu_\e) + \rho_\e \vb f & \text{in } (0,T)\times D_\e,\\
\vu_\e=0 & \text{on } (0,T)\times \del D_\e,\\
\rho_\e(0,\cdot)=\rho_{\e 0},\ (\rho_\e\vu_\e)(0,\cdot)=\vb q_{\e 0} & \text{in } D_\e.
\end{cases}
\end{align}
Here, $\rho_\e$ and $\vb u_\e$ denote the fluid's density and velocity, respectively, $p(s)=s^\gamma$ for some $\gamma>1$, $\mathbb{S}(\nabla \vu)$ is the Newtonian viscous stress tensor of the form
\begin{align*}
\mathbb{S}(\nabla\vu)=\mu(\nabla\vu + \nabla\vu^T-\div(\vb u)\mathbb{I})+\eta\div(\vb u)\mathbb{I}, \quad \mu>0,\, \eta\geq 0,
\end{align*}
and $\vb f\in L^\infty((0,T)\times D)$ is given.

\subsection{Weak solutions and main result}
For further use, we introduce the concept of finite energy weak solutions.
\begin{defin}\label{def1}
Let $T>0$ be fixed, $\gamma>1$, and let the initial data satisfy
\begin{align*}
\rho(0,\cdot)=\rho_0,\quad (\rho\vb u)(0,\cdot)=\vb q_0,
\end{align*}
together with the compatibility conditions
\begin{align}\label{init}
\begin{split}
&\rho_0\geq 0 \text{ a.e.~in } D_\e, \quad \rho_0\in L^\gamma(D_\e),\\
&\vb q_0=0 \text{ on } \{\rho_0=0\}, \quad \vb q_0 \in L^\frac{2\gamma}{\gamma+1}(D_\e),\quad \frac{|\vb q_0|^2}{\rho_0} \in L^1(D_\e).
\end{split}
\end{align}
 We call a duplet $(\rho,\vb u)$ a \emph{finite energy weak solution} to system \eqref{NSE} if:
\begin{itemize}
\item The solution belongs to the regularity class
\begin{gather*}
\rho\geq 0 \text{ a.e.~in } D_\e, \quad \rho\in L^\infty(0,T;L^\gamma(D_\e)), \quad \int_{D_\e} \rho \dd x = \int_{D_\e} \rho_0 \dd x,\\
\vb u\in L^2(0,T;W_0^{1,2}(D_\e)), \quad \rho\vb u\in L^\infty(0,T;L^\frac{2\gamma}{\gamma+1}(D_\e));
\end{gather*}
\item We have
\begin{align}\label{renCE}
\del_t b(\tilde{\rho})+\div(b(\tilde{\rho})\tilde{\vu})+(\tilde{\rho}b'(\tilde{\rho})-b(\tilde{\rho}))\div\tilde{\vu}=0 \text{ in } \mathcal{D}'((0,T)\times \R^2)
\end{align}
for any $b\in C^1([0,\infty))$;
\item For any $\phi\in C_c^\infty([0,T)\times D_\e; \R^2)$,
\begin{align}\label{wkMom}
&\int_0^T\int_{D_\e} \rho \vb u \cdot \del_t \phi \dd x \dd t + \int_0^T\int_{D_\e} \rho \vb u\otimes \vb u : \nabla \phi \dd x \dd t + \int_0^T\int_{D_\e} \rho^\gamma\div \phi \dd x \dd t \\
&- \int_0^T\int_{D_\e} \mathbb{S}(\nabla \vb u):\nabla \phi \dd x \dd t + \int_0^T\int_{D_\e} \rho \vb f \cdot \phi \dd x \dd t = -\int_{D_\e} \vb q_0 \cdot \phi(0,\cdot) \dd x;
\end{align}

\item For almost any $\tau \in [0,T]$, the energy inequality holds:
\begin{align}\label{EI}
&\int_{D_\e} \frac12 \rho |\vb u|^2(\tau,\cdot) + \frac{\rho^\gamma(\tau,\cdot)}{\gamma-1} \dd x + \int_0^\tau \int_{D_\e} \mathbb{S}(\nabla\vb u):\nabla\vb u \dd x \dd t \\
&\quad \leq \int_{D_\e} \frac{|\vb q_0|^2}{2\rho_0} + \frac{\rho_0^\gamma}{\gamma-1} \dd x + \int_0^\tau \int_{D_\e} \rho \vb f\cdot\vb u \dd x \dd t.
\end{align}
\end{itemize}
\end{defin}

Regarding existence of weak solutions, we have the following
\begin{theorem}[{\cite[Theorem~1.1 and Section~5]{FeireislNovotnyPetzeltova2001}}]
Let $D_\e\subset \R^2$ be a bounded domain with smooth boundary, $\gamma>1$, $T>0$ be given. Let the initial data satisfy \eqref{init}. Then, there exists a finite energy weak solution $(\rho,\vu)$ to system \eqref{NSE} in the sense of Definition~\ref{def1}.
\end{theorem}

We are now in the position to state our main result in this paper.
\begin{theorem}\label{thm1}
Let $D\subset\R^2$ be a bounded domain with smooth boundary, $\{z_i^\e\}_{i\in\N}\subset\R^2$ be a collection of points satisfying \eqref{defKeps} and \eqref{scKe}, $\alpha>2$, and $D_\e$ be defined as in \eqref{defDeps}. Let $\gamma>2$, $(\rho_\e,\vb u_\e)$ be a sequence of finite energy weak solutions to system \eqref{NSE} emanating from the initial data $(\rho_{\e0}, \vb q_{\e0})$, and assume
\begin{align}\label{initConv}
\tilde{\rho}_{\e0}\to \rho_0 \text{ in } L^\gamma(D), \quad  \frac{|\tilde{\vb q}_{\e0}|^2}{\tilde{\rho}_{\e0}} \to \frac{|\vb q_0|^2}{\rho_0} \text{ in } L^1(D).
\end{align}
Then, there exists a subsequence (not relabelled) such that
\begin{align*}
\tilde{\rho}_\e\weak^* \rho \text{ weakly$^*$ in } L^\infty(0,T;L^\gamma(D)),\quad \tilde{\vu}_\e\weak \vu \text{ weakly in } L^2(0,T;W_0^{1,2}(D)),
\end{align*}
where $(\rho,\vu)$ is a solution to system \eqref{NSE} in the domain $(0,T)\times D$ with initial conditions $\rho(0)=\rho_0$ and $(\rho\vb u)(0)=\vb q_0$.
\end{theorem}

The restriction $\gamma>2$ is necessary to ensure that the convective term converges in the right way, see Remark \ref{rem1}. As a matter of fact, this restriction seems to be optimal in the sense that it cannot be lowered by the method presented here (see also \cite[Section~7]{OschmannPokorny2023} for higher dimensions $d \geq 3$).

\section{Uniform bounds}\label{sec:Bds}
\begin{lemma}
Under the assumptions of Theorem \ref{thm1}, we have
\begin{align*}
\|\rho_\e\|_{L^\infty(0,T;L^\gamma(D_\e))} + \|\sqrt{\rho_\e}\vb u_\e\|_{L^\infty(0,T;L^2(D_\e))} + \|\vb u_\e\|_{L^2(0,T;W_0^{1,2}(D_\e))} \leq C
\end{align*}
for some constant $C>0$ independent of $\e$.
\end{lemma}
\begin{proof}
By the energy inequality \eqref{EI} and the assumptions on the initial data \eqref{initConv}, we obtain
\begin{align*}
\int_{D_\e} \frac12 \rho_\e |\vb u_\e|^2(\tau,\cdot) + \frac{\rho_\e^\gamma(\tau,\cdot)}{\gamma-1} \dd x + \int_0^\tau \int_{D_\e} \mathbb{S}(\nabla\vb u_\e):\nabla\vb u_\e \dd x \dd t \leq C + \int_0^\tau\int_{D_\e} \rho_\e \vb f \cdot \vb u_\e \dd x \dd t.
\end{align*}
Note further that the conservation of mass and the convergence of the initial data $\rho_{\e0}$ yields
\begin{align*}
\|\rho_\e\|_{L^\infty(0,T;L^1(D_\e))}=\|\rho_{\e0}\|_{L^1(D_\e)}\leq |D_\e|^{1-\frac{1}{\gamma}} \|\tilde{\rho}_{\e0}\|_{L^\gamma(D)}\leq C
\end{align*}
since $|D\setminus D_\e|\leq C \e^{-2} a_\e^2\to 0$, hence $|D_\e|\leq C$. Using now H\"older's and Young's inequality, we get for almost any $\tau\in [0,T]$
\begin{align*}
\int_{D_\e} \rho_\e\vb f \cdot \vb u_\e(\tau) \dd x \dd t \leq C \|\rho_\e(\tau)\|_{L^1(D_\e)}^\frac12 \|\rho_\e |\vb u_\e|^2(\tau)\|_{L^1(D_\e)}^\frac12 \leq C + \frac12 \|\rho_\e |\vb u_\e|^2(\tau)\|_{L^1(D_\e)}.
\end{align*}
Thus, we end up with the inequality
\begin{align*}
\int_{D_\e} \frac12 \rho_\e |\vb u_\e|^2(\tau,\cdot) + \frac{\rho_\e^\gamma(\tau,\cdot)}{\gamma-1} \dd x + \int_0^\tau \int_{D_\e} \mathbb{S}(\nabla\vb u_\e):\nabla\vb u_\e \dd x \dd t \leq C + \int_0^\tau \int_{D_\e} \frac12 \rho_\e |\vu_\e|^2 \dd x \dd t.
\end{align*}
Using finally Gr\"onwall's, Korn's, and Poincar\'e's inequality, we conclude easily.
\end{proof}

Moreover, we can improve the pressure regularity, the proof of which follows the same lines as \cite[Appendix~B]{Bravin2022} (see also Appendix \ref{apA}).
\begin{lemma}\label{lem:press}
For any $\theta<\gamma-1$,
\begin{align*}
\int_0^T\int_{D_\e} \rho_\e^{\gamma+\theta} \dd x \dd t \leq C.
\end{align*}
\end{lemma}

\section{Suitable test functions}
In order to pass to the limit in the momentum equation, we need an appropriate test function obtained from an arbitrary function $\phi\in C_c^\infty([0,T)\times D)$. To this end, we have to modify $\phi$ such that it vanishes on the holes. As a consequence of our construction, we will make $\phi$ vanish on a slightly larger set, keeping fixed the scaling and number of the holes. First, by the definition of $D_\e$ in \eqref{defDeps} and the holes as $a_\e F$ for some compact set $F\subset B_1(0)$, we have $z_i^\e+a_\e F\subset B_{a_\e}(z_i^\e)$. Now, we define a ``single hole'' cut-off function via
\begin{align}\label{aeps}
\begin{split}
\mathfrak{y}_\e(r)=\begin{cases}
1 & \text{if } 0\leq r < a_\e,\\
\frac{\log(\varpi_\e a_\e)-\log(r)}{\log(\varpi_\e a_\e)-\log(a_\e)} & a_\e \leq r < \varpi_\e a_\e,\\ 0 & \text{else},
\end{cases}
\end{split}
\end{align}
where $1<\varpi_\e\to \infty$ such that $\varpi_\e a_\e\leq \e$ will be chosen later. After passing to radial coordinates and denoting $x_i$ the $i$-th component of the vector $x$, it is easy to see that for any $1\leq q < \infty$, we have
\begin{align}\label{est0}
\begin{split}
\|\mathfrak{y}_\e\|_{L^\infty(\R^2)}+\|\nabla \mathfrak{y}_\e x_i\|_{L^\infty(\R^2)} &\lesssim 1,\\
\|\nabla \mathfrak{y}_\e\|_{L^q(\R^2)}^q + \|\nabla^2 \mathfrak{y}_\e x_i\|_{L^q(\R^2)}^q &\lesssim \begin{cases}
\frac{a_\e^{2-q}}{|\log \varpi_\e|^q} |\varpi_\e^{2-q}-1| & \text{if } q\neq 2,\\
|\log \varpi_\e|^{-1} & \text{if } q=2.
\end{cases}
\end{split}
\end{align}

To define an appropriate cut-off function for multiple holes in the whole of $D$, we follow an idea of Bravin in \cite{Bravin2022} for a single hole in $\R^2$. Recall the definitions of $z_i^\e$ and $K_\e$ in \eqref{defKeps}. We set $\eta_\e^i(x) = \mathfrak{y}_\e(|x-z_i^\e|)$ for $x\in D$, and define the matrix-valued cut-off function
\begin{align*}
\Phi_\e = \mathbb{I}-\sum_{i\in K_\e} \eta_\e^i\mathbb{I} + \nabla^\perp \eta_\e^i\otimes (x-z_i^\e)^\perp,
\end{align*}
where $x^\perp=(-x_2, x_1)$ and $\nabla^\perp=(-\del_2, \del_1)$. Note especially that
\begin{align*}
\div \Phi_\e = -\sum_{i\in K_\e} \div\big(\nabla^\perp [\eta_\e^i(x-z_i^\e)^\perp] \big)=0,
\end{align*}
where the divergence is taken row-wise as $\div \Phi_\e=\div(\Phi_\e^T \vb e_1) \vb e_1 + \div(\Phi_\e^T \vb e_2) \vb e_2$. We summarize the properties of $\Phi_\e$ in the following

\begin{lemma}\label{lem2}
The function $\Phi_\e$ fulfils
\begin{align*}
\Phi_\e &\in W^{1,q}(D)\cap L^\infty(D) \text{ for any } q\geq 1,\\
\Phi_\e &=0 \text{ on } D\setminus D_\e,\\
\Phi_\e &=\mathbb{I} \text{ on } D\setminus \bigcup_{i\in K_\e} B_{\varpi_\e a_\e}(z_i^\e).
\end{align*}
Moreover, for any $1\leq q\leq \infty$,
\begin{align}\label{ph1}
\begin{split}
\|\Phi_\e-\mathbb{I}\|_{L^q(D)} &\lesssim \e^{-\frac2q} (\varpi_\e a_\e)^\frac2q,\\
\|\nabla\Phi_\e\|_{L^2(D)} &\lesssim \e^{-1}|\log \varpi_\e|^{-\frac12},
\end{split}
\end{align}
with the convention $1/\infty=0$. In turn, for any $\phi\in C_c^\infty(D)$ and any $q\leq 2$,
\begin{align}\label{ph2}
\|\nabla (\Phi_\e \phi)-\Phi_\e \nabla\phi\|_{L^q(D)}\lesssim \e^{-1} |\log \varpi_\e|^{-\frac12} \|\phi\|_{L^{\frac{2q}{2-q}}(D)},
\end{align}
with the convention $1/0=\infty$.
\end{lemma}
\begin{proof}
By the definition of $\Phi_\e$, we immediately see that $\Phi_\e=0$ on the holes as well as $\Phi_\e=\mathbb{I}$ outside every $B_{\varpi_\e a_\e} (z_i^\e)$. Further, noticing that the holes are disjoint and their number in $D$ is bounded by $C\e^{-2}$ by \eqref{scKe}, we easily conclude \eqref{ph1} by using \eqref{est0}. Inequality \eqref{ph2} is a direct consequence of H\"older's inequality
\begin{align*}
&\|\nabla (\Phi_\e \phi)-\Phi_\e \nabla\phi\|_{L^q(D)}=\|(\nabla (\Phi_\e \vb e_1) \phi , \nabla (\Phi_\e \vb e_2) \phi)\|_{L^q(D)} \\
&\leq \|\nabla\Phi_\e\|_{L^2(D)} \|\phi\|_{L^{\frac{2q}{2-q}}(D)}\lesssim \e^{-1} |\log \varpi_\e|^{-\frac12} \|\phi\|_{L^{\frac{2q}{2-q}}(D)}.
\end{align*}
\end{proof}

\section{Convergences and equations in homogeneous domain}\label{sec:Conv}
In this section, we show that the functions $\rho_\e$ and $\vu_\e$ converge to functions $\rho$ and $\vu$ in a proper way, respectively, and figure out the limiting system they solve. First, note that the extended functions $\tilde{\rho}_\e$ and $\tilde{\vb u}_\e$ share the same regularity in the whole of $\R^2$ as their originating functions $\rho_\e$ and $\vb u_\e$ in $D_\e$. Hence, from the uniform bounds derived in Section~\ref{sec:Bds}, we have
\begin{align}\label{conv}
\begin{split}
\tilde{\rho}_\e &\weak \rho \text{ weakly in } L^{(2\gamma-1)^-}([0,T)\times D),\\
\tilde{\rho}_\e &\weak^* \rho \text{ weakly$^*$ in } L^\infty(0,T;L^\gamma(D)),\\
\tilde{\vb u}_\e &\weak \vb u \text{ weakly in } L^2(0,T;W_0^{1,2}(D)),
\end{split}
\end{align}
where we denoted by $(2\gamma-1)^-$ any number less than but arbitrarily close to $2\gamma-1$. Our first result concerns the extended continuity equation, which can be proven similarly to \cite[Proposition~3.3]{LuSchwarzacher2018}.
\begin{lemma}\label{lem1}
The extended functions $\tilde{\rho}_\e$ and $\tilde{\vb u}_\e$ fulfil
\begin{align*}
\del_t \tilde{\rho}_\e + \div(\tilde{\rho}_\e \tilde{\vb u}_\e)=0 \text{ in } \mathcal{D}'((0,T)\times \R^2).
\end{align*}
\end{lemma}

Let us moreover show that
\begin{align}\label{eq:CV}
\begin{split}
\Phi_\e^T \tilde{\rho}_\e \tilde{\vb u}_\e &\to \rho \vb u \text{ in } C(0,T;L_{\rm weak}^r(D)), \quad r < \frac{2\gamma}{\gamma+1},\\
\Phi_\e^T \tilde{\rho}_\e \tilde{\vb u}_\e \otimes \tilde{\vb u}_\e &\to \rho \vb u \otimes \vb u \text{ in } \mathcal{D}'((0,T)\times D).
\end{split}
\end{align}
Indeed, from the uniform bounds on $\tilde{\rho}_\e$ and $\tilde{\vb u}_\e$ derived in Section~\ref{sec:Bds}, we have
\begin{align*}
\|\tilde{\rho}_\e\tilde{\vb u}_\e\|_{L^\infty(0,T;L^\frac{2\gamma}{\gamma+1}(D))}\leq \|\sqrt{\tilde{\rho}_\e}\|_{L^\infty(0,T;L^{2\gamma}(D))} \|\sqrt{\tilde{\rho}_\e}\tilde{\vb u}_\e\|_{L^\infty(0,T;L^2(D))}\leq C.
\end{align*}
Moreover, by Lemma~\ref{lem1}, we have
\begin{align*}
\del_t \tilde{\rho}_\e \text{ bounded in } L^2(0,T;W^{-1,p}(D)) \text{ for any } p<\gamma.
\end{align*}
Applying \cite[Lemma~5.1]{Lions1998} now shows
\begin{align*}
\tilde{\rho}_\e\tilde{\vb u}_\e \to \rho\vb u \text{ in } \mathcal{D}'((0,T)\times D).
\end{align*}
Furthermore, an Aubin-Lions type argument shows
\begin{align}\label{AL}
\tilde{\rho}_\e \to \rho \text{ in } C(0,T;L_{\rm weak}^\gamma(D)),\quad \tilde{\rho}_\e \tilde{\vb u}_\e \to \rho \vb u \text{ in } C(0,T;L_{\rm weak}^\frac{2\gamma}{\gamma+1}(D)).
\end{align}
A similar argument applies to $\tilde{\rho}_\e\tilde{\vb u}_\e\otimes\tilde{\vb u}_\e$. Since $\Phi_\e$ converges in any $L^q(D)$, we conclude \eqref{eq:CV}.\\

\subsection{Limit in the continuity equation}
From Lemma~\ref{lem1}, we obtain for any $\psi\in C_c^\infty([0,T)\times \R^2)$
\begin{align*}
\int_0^T \int_{\R^2} \tilde{\rho}_\e\del_t \psi \dd x \dd t + \int_0^T\int_{\R^2} \tilde{\rho}_\e \tilde{\vb u}_\e \cdot \nabla \psi \dd x \dd t = -\int_{\R^2} \tilde{\rho}_{\e0} \psi(0,\cdot) \dd x.
\end{align*}
Together with the assumptions on the initial data \eqref{initConv}, and the convergences \eqref{conv} and \eqref{AL}, we pass with $\e\to 0$ in the above equation to obtain
\begin{align*}
\del_t \rho + \div(\rho \vb u)=0 \text{ in } \mathcal{D}'((0,T)\times \R^2).
\end{align*}
According to \cite[Lemma 6.9]{Novotny2004}, this shows that the couple $(\rho,\vb u)$ also fulfils the renormalized continuity equation \eqref{renCE}.

\subsection{Limit in the momentum equation}
To pass to the limit in the weak formulation of the momentum equation \eqref{wkMom}, we use $\Phi_\e \phi\in C_c^\infty([0,T)\times D_\e)$ as a proper test function. Recalling that $\Phi_\e=0$ on the holes, we can extend $\vb q_{\e0}$, $\rho_\e$, and $\vb u_\e$ by zero to the whole of $D$, leading to
\begin{align*}
0&=\int_D \tilde{\vb q}_{\e0}\cdot  \Phi_\e \phi(0,\cdot) \dd x + \int_0^T\int_D \tilde{\rho}_\e\tilde{\vb u}_\e \cdot \Phi_\e \del_t\phi \dd x \dd t + \int_0^T\int_D \tilde{\rho}_\e \tilde{\vb u}_\e \otimes \tilde{\vb u}_\e : \nabla (\Phi_\e \phi) \dd x \dd t\\
&\quad + \int_0^T\int_D \tilde{\rho}_\e^\gamma \div(\Phi_\e \phi) \dd x \dd t - \int_0^T \int_D \mathbb{S}(\nabla\tilde{\vb u}_\e):\nabla(\Phi_\e \phi) \dd x \dd t + \int_0^T\int_D \tilde{\rho}_\e \vb f\cdot\Phi_\e \phi \dd x \dd t\\
&= \sum_{j=1}^6 I_j.
\end{align*}

We will pass with $\e\to 0$ in each integral separately. To this end, we need to choose $\varpi_\e$ from \eqref{aeps} in a proper way. We want that $\Phi_\e \to \mathbb{I}$ strongly in $L^q(D)$ for any $1\leq q<\infty$. According to \eqref{ph1} in Lemma~\ref{lem2}, we may choose $\varpi_\e$ such that
\begin{align}\label{chosAlp}
\e^{-1} \varpi_\e a_\e=\e^\delta \text{ for some } \delta>0, \quad \text{that is,} \quad \varpi_\e = \e^{1+\delta} a_\e^{-1}.
\end{align}
We remark that this choice is much faster growing than the requirement made in \cite[Proposition~1]{Bravin2022}. Note also that this yields
\begin{align*}
\|\nabla \Phi_\e\|_{L^2(D)}^2 \lesssim \e^{-2} |\log \varpi_\e|^{-1}\lesssim \e^{\alpha-2},
\end{align*}
which is the critical scaling in our setting $\alpha>2$. The precise choice of $\delta$ depends on the estimates for the Bogovski\u{\i} operator in order to have a higher integrability for the density; see Appendix \ref{apA} for details.
\\

Now, for $I_1$, we obtain
\begin{align*}
\int_D \tilde{\vb q}_{\e0} \cdot \Phi_\e \phi(0,\cdot) \dd x &= \int_D \frac{\tilde{\vb q}_{\e0}}{\sqrt{\tilde{\rho}_{\e0}}} \sqrt{\tilde{\rho}_{\e0}} \cdot \Phi_\e \phi(0,\cdot) \dd x\\
&\to \int_D \frac{\vb q_0}{\sqrt{\rho_0}} \sqrt{\rho_0} \cdot \phi(0,\cdot) \dd x = \int_D \vb q_0 \cdot \phi(0,\cdot) \dd x,
\end{align*}
where we used that $\tilde{\vb q}_{\e0}/\sqrt{\tilde{\rho}_{\e0}} \to \vb q_0/\sqrt{\rho_0}$ strongly in $L^2(D)$ and $\sqrt{\tilde{\rho}_{\e0}} \to \sqrt{\rho_0}$ strongly in $L^{2\gamma}(D)$.\\

For $I_2$, we get with \eqref{eq:CV} the convergence
\begin{align*}
\int_0^T \int_D \tilde{\rho}_\e \tilde{\vb u}_\e \cdot \Phi_\e \del_t \phi \dd x \dd t = \int_0^T \int_D \Phi_\e^T \tilde{\rho}_\e \tilde{\vb u}_\e \cdot \del_t \phi \dd x \dd t \to \int_0^T \int_D \rho \vb u \cdot \del_t \phi \dd x \dd t.
\end{align*}

For the pressure term $I_4$, recall that the function $\Phi_\e$ is divergence-free. Thus,
\begin{align*}
\int_0^T\int_D \tilde{\rho}_\e^\gamma \div(\Phi_\e\phi) \dd x \dd t &= \int_0^T\int_D \tilde{\rho}_\e^\gamma \Phi_\e:\nabla\phi  \dd x \dd t\\
\to \int_0^T\int_D \overline{\rho^\gamma} \mathbb{I}:\nabla\phi \dd x \dd t &= \int_0^T\int_D \overline{\rho^\gamma} \div\phi \dd x \dd t,
\end{align*}
where we denoted by $\overline{\rho^\gamma}$ the weak limit of $\tilde{\rho}_\e^\gamma$ in $L^\frac{(2\gamma-1)^-}{\gamma}((0,T)\times D)$.\\

To pass to the limit in the diffusive term $I_5$, we rewrite
\begin{align*}
\int_0^T\int_D \mathbb{S}(\nabla\tilde{\vb u}_\e):\nabla(\Phi_\e\phi) \dd x \dd t &= \int_0^T\int_D \mathbb{S}(\nabla\tilde{\vb u}_\e):(\Phi_\e \nabla \phi) \dd x \dd t \\
&\quad + \int_0^T\int_D \mathbb{S}(\nabla\tilde{\vb u}_\e):(\nabla(\Phi_\e \phi)-\Phi_\e\nabla\phi) \dd x \dd t.
\end{align*}
The latter term converges to zero due to
\begin{align*}
\bigg|\int_0^T\int_D \mathbb{S}(\nabla \tilde{\vb u}_\e) : (\nabla(\Phi_\e \phi)-\Phi_\e\nabla\phi) \dd x \dd t \bigg| &\lesssim \|\nabla \tilde{\vb u}_\e\|_{L^2(0,T;L^2(D))} \|\nabla(\Phi_\e \phi)-\Phi_\e\nabla\phi\|_{L^\infty(0,T;L^2(D))}\\
& \lesssim \|\nabla \Phi_\e\|_{L^2(D)} \|\phi\|_{L^\infty((0,T)\times D)} \lesssim \e^\frac{\alpha-2}{2}\|\phi\|_{L^\infty((0,T)\times D)}.
\end{align*}
Together with the strong convergence of $\Phi_\e \to \mathbb{I}$ in $L^2(D)$ and the weak convergence of $\nabla\tilde{\vb u}_\e \weak \nabla\vb u$ in $L^2(0,T;L^2(D))$, we deduce
\begin{align*}
\int_0^T\int_D \mathbb{S}(\nabla\tilde{\vb u}_\e):\nabla(\Phi_\e\phi) \to \int_0^T\int_D \mathbb{S}(\nabla\vb u):\nabla\phi.
\end{align*}

For the force term $I_6$,
\begin{align*}
\int_0^T\int_D \tilde{\rho}_\e\vb f\cdot \Phi_\e \phi \dd x \dd t \to \int_0^T\int_D \rho \vb f\cdot\phi \dd x \dd t
\end{align*}
by the strong convergence of $\Phi_\e$ to $\mathbb{I}$ in any $L^q(D)$.\\

Let us turn to $I_3$, where we argue similar as for $I_5$. We rewrite
\begin{align*}
\int_0^T\int_D \tilde{\rho}_\e \tilde{\vb u}_\e \otimes \tilde{\vb u}_\e : \nabla (\Phi_\e \phi) \dd x \dd t &= \int_0^T\int_D \tilde{\rho}_\e \tilde{\vb u}_\e \otimes \tilde{\vb u}_\e : (\Phi_\e \nabla \phi) \dd x \dd t\\
&\quad  + \int_0^T\int_D \tilde{\rho}_\e \tilde{\vb u}_\e \otimes \tilde{\vb u}_\e : (\nabla(\Phi_\e\phi)-\Phi_\e\nabla\phi) \dd x \dd t\\
&= \int_0^T\int_D \Phi_\e^T \tilde{\rho}_\e \tilde{\vb u}_\e \otimes \tilde{\vb u}_\e : \nabla \phi \dd x \dd t\\
&\quad  + \int_0^T\int_D \tilde{\rho}_\e \tilde{\vb u}_\e \otimes \tilde{\vb u}_\e : (\nabla(\Phi_\e\phi)-\Phi_\e\nabla\phi) \dd x \dd t.
\end{align*}
The latter term vanishes due to the embedding $W^{1,2}(D)\subset L^p(D)$ for any $p<\infty$. Indeed, we get with $\gamma>2$ and the uniform bounds on $\rho_\e$ and $\vu_\e$
\begin{align*}
&\bigg| \int_0^T\int_D \tilde{\rho}_\e\tilde{\vb u}_\e\otimes \tilde{\vb u}_\e : (\nabla(\Phi_\e\phi)-\Phi_\e\nabla\phi) \bigg|\\
&\leq \|\tilde{\rho}_\e\|_{L^\infty(0,T;L^\gamma(D))} \|\tilde{\vb u}_\e\|_{L^2(0,T;L^\frac{4\gamma}{\gamma-2}(D))}^2 \|\nabla(\Phi_\e\phi)-\Phi_\e\nabla\phi\|_{L^\infty(0,T;L^2(D))}\\
&\lesssim \e^\frac{\alpha-2}{2} \|\phi\|_{L^\infty((0,T)\times D)}.
\end{align*}
Hence, by $\Phi_\e^T \tilde{\rho}_\e \tilde{\vb u}_\e \otimes \tilde{\vb u}_\e \to \rho \vb u \otimes \vb u$ in $\mathcal{D}'((0,T)\times D)$, we obtain
\begin{align*}
\int_0^T\int_D \tilde{\rho}_\e \tilde{\vb u}_\e \otimes \tilde{\vb u}_\e : \nabla (\Phi_\e \phi) \dd x \dd t \to \int_0^T\int_D \rho\vb u\otimes\vb u : \nabla \phi \dd x \dd t.
\end{align*}
Collecting all convergences above, we end up with
\begin{align*}
0&=\int_D \vb q_0\cdot  \phi(0,\cdot) \dd x + \int_0^T\int_D \rho\vb u \cdot \del_t\phi \dd x \dd t + \int_0^T\int_D \rho \vb u \otimes \vb u : \nabla \phi \dd x \dd t\\
&\quad + \int_0^T\int_D \overline{\rho^\gamma} \div\phi \dd x \dd t - \int_0^T \int_D \mathbb{S}(\nabla\vb u):\nabla\phi \dd x \dd t + \int_0^T\int_D \rho \vb f\cdot\phi \dd x \dd t.
\end{align*}
In order to finish the proof of Theorem~\ref{thm1}, we have to show that $\overline{\rho^\gamma}=\rho^\gamma$, which can be done similarly to \cite[Section~7]{Bravin2022} and the arguments given in \cite[Lemma~4.5]{DieningFeireislLu2017}.

\begin{rem}\label{rem1}
As already mentioned, the stronger assumption $\gamma>2$ is needed to ensure that $\Phi_\e^T\tilde{\rho}_\e\tilde{\vu}_\e\otimes\tilde{\vu}_\e$ converges to its counterpart. Indeed, for $1<\gamma\leq 2$, we need to ensure that the term $\nabla(\Phi_\e\phi)-\Phi_\e\nabla\phi=(\nabla(\Phi_\e\vb e_1)\phi, \nabla(\Phi_\e \vb e_2)\phi)$ vanishes in $L^\infty(0,T;L^q(D))$ for some $q>2$. A similar calculation as for the $L^2$-setting shows for any $q>2$
\begin{align*}
\|\nabla\Phi_\e\|_{L^q(D)}^q\lesssim \frac{\e^{-2} a_\e^{2-q} (1-\varpi_\e^{2-q})}{|\log \varpi_\e|^q} \leq \frac{\e^{-2} a_\e^{2-q}}{|\log \varpi_\e|^q}.
\end{align*}
In order to make this vanish, one would need $\varpi_\e \sim \exp(a_\e^{-1}) = \exp(\exp(\e^{-\alpha}))$. However, this yields $\varpi_\e a_\e\to \infty$, meaning that neither $\Phi_\e \to \mathbb{I}$ in some $L^q(D)$ nor that the balls $B_{\varpi_\e a_\e}(z_i^\e)$ stay inside $D$.
\end{rem}

\appendix
\section{Bogovski\u{\i}'s operator in 2D and improved pressure estimates}\label{apA}
In this section, we give an inverse to the divergence in the two-dimensional perforated domain, and estimate its norm in any $L^q(D_\e)$. To the best of the authors' knowledge, such estimates for two spatial dimensions are just known in the $L^2$-setting, see \cite[Section~1.4]{NecasovaPan2022}. Therefore, we give here an explicit proof, which might be of independent interest. As an application, we explain how to use it to prove Lemma~\ref{lem:press}.

\begin{theorem}
Let $D\subset\R^2$ be a bounded domain with smooth boundary and $D_\e$ be defined as in \eqref{defDeps}. Then, there exists an operator $\B_\e$ such that for any $q\geq 1$,
\begin{align*}
\B_\e:L_0^q(D_\e)=\{f\in L^q(D_\e):\int_{D_\e} f \dd x =0\}\to W_0^{1,q}(D_\e;\R^2),
\end{align*}
and for any $f\in L_0^q(D_\e)$ we have
\begin{align*}
\div \B_\e(f)=f \text{ in } D_\e,\quad \|\B_\e\|_{W_0^{1,q}(D_\e)}^q \lesssim (1+C(\e,q))\|f\|_{L^q(D_\e)}^q,
\end{align*}
where
\begin{align*}
C(\e,q)=\varpi_\e^{-2} a_\e^{-q} \begin{cases}
|\log \varpi_\e|^{-q} |\varpi_\e^{2-q}-1| & \text{if } q\neq 2,\\
|\log \varpi_\e|^{-1} & \text{if } q=2,
\end{cases}
\end{align*}
and $\varpi_\e$ is as in \eqref{chosAlp}.
\end{theorem}
\begin{proof}
We follow the idea of \cite{DieningFeireislLu2017}, where $L^q$-estimates are given for the case of three spatial dimensions. Let $f\in L_0^q(D_\e)$. Then, there exists a function $\vb u\in W_0^{1,q}(D)$ such that
\begin{align*}
\div\vb u=\tilde{f} \text{ in } D, \quad \|\vb u\|_{W_0^{1,q}(D)}\leq C \|f\|_{L^q(D_\e)}
\end{align*}
for some constant $C>0$ independent of $\e$ (see \cite{Bogovskii1980, Galdi2011}). However, $\vb u$ does not vanish on the holes in general. To overcome this, note that the domains $B_{\varpi_\e a_\e}(z_i^\e)\setminus B_{a_\e}(z_i^\e)$ are uniform John domains (see, e.g., \cite[Example~3.2.2]{OschmannDiss2022}), so for each $i\in K_\e$, there exists a Bogovski\u{\i} operator $\B_{i,\e}$ satisfying
\begin{align*}
&\B_{i,\e}:L_0^q(B_{\varpi_\e a_\e}(z_i^\e)\setminus B_{a_\e}(z_i^\e))\to W_0^{1,q}(B_{\varpi_\e a_\e}(z_i^\e)\setminus B_{a_\e}(z_i^\e)),\\
&\div B_{i,\e}(g)=g \text{ in } B_{\varpi_\e a_\e}(z_i^\e)\setminus B_{a_\e}(z_i^\e),\\
&\|\B_{i,\e}(g)\|_{L^q(B_{\varpi_\e a_\e}(z_i^\e)\setminus B_{a_\e}(z_i^\e))}\leq C \|g\|_{L^q(B_{\varpi_\e a_\e}(z_i^\e)\setminus B_{a_\e}(z_i^\e))}
\end{align*}
for some constant $C>0$ independent of $\e$ (see \cite[Theorem~5.2]{DieningRuzickaSchumacher2010}). Furthermore, we define $\mathfrak{y}_\e$ as in \eqref{aeps}, and
\begin{align*}
\vartheta_\e(r)=\begin{cases}
1 & \text{if } 0\leq r < \varpi_\e a_\e/2,\\
\frac{2}{\varpi_\e a_\e}(\varpi_\e a_\e-r) & \text{if } \varpi_\e a_\e/2\leq r < \varpi_\e a_\e,\\
0 & \text{else}.
\end{cases}
\end{align*}
As before, set for $z_i^\e\in K_\e$ and $x\in D$ the functions $\eta_\e^i(x)=\mathfrak{y}_\e(|x-z_i^\e|)$ and $\theta_\e^i(x)=\vartheta_\e(|x-z_i^\e|)$, and define for $\vb u\in W^{1,q}(B_{\varpi_\e a_\e}(z_i^\e))$ the operator $L_{i,\e}$ as
\begin{align*}
L_{i,\e} \vb u(x) = \theta_\e^i(x)\bigg(\vb u(x)-\frac{1}{|B_{\varpi_\e a_\e}(z_i^\e)|} \int_{B_{\varpi_\e a_\e}(z_i^\e)} \vb u \dd x \bigg) + \eta_\e^i(x)\frac{1}{|B_{\varpi_\e a_\e}(z_i^\e)|} \int_{B_{\varpi_\e a_\e}(z_i^\e)} \vb u \dd x.
\end{align*}
Note that this immediately implies $L_{i,\e}\vb u=0$ on $\del B_{\varpi_\e a_\e}(z_i^\e)$ as well as $L_{i,\e}\vb u=\vb u$ on $\del B_{a_\e}(z_i^\e)$. Moreover, by Poincar\'e's inequality,
\begin{align*}
\bigg\|\vb u-\frac{1}{|B_{\varpi_\e a_\e}(z_i^\e)|} \int_{B_{\varpi_\e a_\e}(z_i^\e)} \vb u \dd x\bigg\|_{L^q(B_{\varpi_\e a_\e}(z_i^\e))} \leq C \varpi_\e a_\e \|\nabla\vb u\|_{L^q(B_{\varpi_\e a_\e}(z_i^\e))}
\end{align*}
for some constant $C>0$ independent of $\e$. Hence, by the estimate \eqref{est0} and H\"older's inequality,
\begin{align*}
&\|\nabla L_{i,\e}\vb u\|_{L^q(B_{\varpi_\e a_\e}(z_i^\e))} \leq \|\nabla\theta_\e^i\|_{L^\infty(D)} \bigg\|\vb u-\frac{1}{|B_{\varpi_\e a_\e}(z_i^\e)|} \int_{B_{\varpi_\e a_\e}(z_i^\e)} \vb u \dd x\bigg\|_{L^q(B_{\varpi_\e a_\e}(z_i^\e))}\\
&\quad + \|\nabla\vb u\|_{L^q(B_{\varpi_\e a_\e}(z_i^\e))} + \|\nabla\eta_\e^i\|_{L^q(B_{\varpi_\e a_\e}(z_i^\e))} \frac{1}{|B_{\varpi_\e a_\e}(z_i^\e)|} \int_{B_{\varpi_\e a_\e}(z_i^\e)} |\vb u| \dd x\\
&\lesssim \|\nabla\vb u\|_{L^q(B_{\varpi_\e a_\e}(z_i^\e))} + (\varpi_\e a_\e)^{-\frac2q} a_\e^{\frac2q-1} \|\vb u\|_{L^q(B_{\varpi_\e a_\e}(z_i^\e))} \begin{cases}
|\log \varpi_\e|^{-1} |\varpi_\e^{2-q}-1|^\frac1q & \text{if } q\neq 2,\\
|\log\varpi_\e|^{-\frac12} & \text{if } q=2,
\end{cases}
\end{align*}
yielding an operator $L_{i,\e}:W^{1,q}(B_{\varpi_\e a_\e}(z_i^\e))\to W_0^{1,q}(B_{\varpi_\e a_\e}(z_i^\e))$. Eventually, we define
\begin{align*}
\B_\e(f)=\vb u - \sum_{i\in K_\e} L_{i,\e}\vb u - \B_{i,\e} \div L_{i,\e} \vb u.
\end{align*}
Note that this operator is well defined due to
\begin{align*}
\int_{B_{\varpi_\e a_\e}(z_i^\e)} \div L_{i,\e}\vb u \dd x = \int_{\del B_{\varpi_\e a_\e}(z_i^\e)} L_{i,\e}\vb u \cdot \vb n \dd \sigma = 0
\end{align*}
since $L_{i,\e}\vb u=0$ on $\del B_{\varpi_\e a_\e}(z_i^\e)$. Furthermore, for any $x\in B_{a_\e}(z_i^\e)$,
\begin{align*}
\div L_{i,\e}\vb u(x)=\div \vb u(x)=\tilde{f}(x)=0
\end{align*}
by $L_{i,\e}\vb u=\vb u$ in $B_{a_\e}(z_i^\e)$ and $\tilde{f}(x)=0$ on $D\setminus D_\e$. Hence,
\begin{align*}
\int_{B_{\varpi_\e a_\e}(z_i^\e)\setminus B_{a_\e}(z_i^\e)} \div L_{i,\e}\vb u\dd x = 0
\end{align*}
as wished. Moreover, this leads for any $x\in B_{a_\e}(z_i^\e)$ to
\begin{align*}
\B_\e(f)(x)=\vb u(x)-L_{i,\e}\vb u(x)=0,
\end{align*}
so indeed $\B_\e(f)=0$ on $D\setminus D_\e$. Seeing finally that the holes $B_{\varpi_\e a_\e}(z_i^\e)$ are disjoint, we sum up the estimates obtained to finish the proof of the theorem.
\end{proof}

With the help of the operator $\B_\e$, we can show Lemma~\ref{lem:press}. Recalling $a_\e=\exp(-\e^{-\alpha})$ for some $\alpha>2$ and $\varpi_\e = \e^{1+\delta}a_\e^{-1}$ from \eqref{chosAlp}, we have for any $1\leq q<2$ the bound
\begin{align*}
1+C(\e,q) &\lesssim 1 + \varpi_\e^{-2} a_\e^{-q} |\log \varpi_\e|^{-q} \varpi_\e^{2-q} = 1+(a_\e \varpi_\e |\log \varpi_\e|)^{-q} \\
&\lesssim 1+\e^{q(\alpha-1-\delta)} \lesssim 1,
\end{align*}
which is uniform as long as $\delta\le \alpha-1$. Similarly, for $q=2$, we have
\begin{align*}
1+C(\e,2)\lesssim 1+(\varpi_\e a_\e)^{-2} |\log \varpi_\e|^{-1}\lesssim 1+\e^{\alpha-2(1+\delta)}\lesssim 1
\end{align*}
as long as $\delta\leq \frac{\alpha-2}{2}$.
The idea is now to test the momentum equation \eqref{wkMom} by the function
\begin{align*}
\phi(t,x)=\xi(t)\B_\e \bigg[\rho_\e^\theta-\frac{1}{|D_\e|}\int_{D_\e} \rho_\e^\theta\bigg]
\end{align*}
for $\theta<\gamma-1$ and some $\xi\in C_c^\infty([0,T))$. Note that the function $\phi$ is not regular enough in the time variable to use it as test function, however, one can overcome this by using a time-regularization argument (see \cite[Section~2.2.5]{FeireislNovotny2009} for details). The proof of the improved integrability of the density now follows the same lines as \cite[Appendix~B]{Bravin2022} (see also \cite[Section~4.2.2]{OschmannDiss2022}).

\section*{Acknowledgement}
{\it \v S. N. and F. O. have been supported by the Czech Science Foundation (GA\v CR) project 22-01591S.  Moreover, \it \v S. N.  has been supported by  Praemium Academi{\ae} of \v S. Ne\v casov\' a. The Institute of Mathematics, CAS is supported by RVO:67985840.}


\bibliographystyle{amsalpha}

\end{document}